\documentclass[reqno, a4paper, 10pt]{amsart}
\usepackage{amssymb}
\usepackage{mathrsfs}
\usepackage{amssymb, url, color, graphicx, amscd, mathrsfs}
\usepackage[colorlinks=true, bookmarks=true, pdfstartview=FitH, pagebackref=true, linktocpage=true, linkcolor = magenta, citecolor = blue]{hyperref}
\usepackage{graphicx}
\usepackage{times}
\newtheorem{theorem}{\bf Theorem}[section]
\newtheorem{define}[theorem]{\bf Definition}
\newtheorem{corollary}[theorem]{\bf Corollary}
\newtheorem{lemma}[theorem]{\bf Lemma}

\newtheorem{remark}{\bf Remark}

\newtheorem*{theorem*}{Theorem}

\numberwithin{equation}{section}

\newcommand{\RR}{\mathbb R}
\newcommand{\intl}{\int\limits}

\DeclareMathOperator{\RicH}{{\rm Ric}_{-}^K}
\allowdisplaybreaks
\begin{document}

\title[Gradient estimates for  weighted $p$-Laplacian equations]{Gradient estimates for  weighted $p$-Laplacian equations on Riemannian manifolds with a Sobolev inequality and integral Ricci bounds}
\def\cfac#1{\ifmmode\setbox7\hbox{$\accent"5E#1$}\else\setbox7\hbox{\accent"5E#1}\penalty 10000\relax\fi\raise 1\ht7\hbox{\lower1.0ex\hbox to 1\wd7{\hss\accent"13\hss}}\penalty 10000\hskip-1\wd7\penalty 10000\box7 }

\author[L. V. Dai]{Le Van Dai}
\address[L. V. Dai]{Department of Mathematics, Mechanics and Informatics, College of Science\\ Vi\^{e}t Nam National University, Ha N\^{o}i, Vi\^{e}t Nam}
\email{\href{mailto: Le Van Dai <daik55maths142@gmail.com>}{daik55maths142@gmail.com}}

\author[N.T. Dung]{Nguyen Thac Dung}
\address[N.T. Dung]{Department of Mathematics, Mechanics and Informatics, College of Science\\ Vi\^{e}t Nam National University, Ha N\^{o}i, Vi\^{e}t Nam}
\email{\href{mailto: N.T. Dung
<dungmath@gmail.com>}{dungmath@gmail.com}}

\author[N. D. Tuyen]{Nguyen Dang Tuyen}
\address[N. D. Tuyen]{Department of Mathematics,\\ National University of Civil Engineering,
Hanoi, Vietnam.}
\email{\href{mailto: Nguyen Dang Tuyen <tuyennd@nuce.edu.vn>}{tuyennd@nuce.edu.vn}}

\author[Liang Zhao]{Liang Zhao}
\address[Liang Zhao]{Department of Mathematics,\\ 
Nanjing University of Aeronautics and Astronautics, Nanjing, China}
\email{\href{mailto: Zhao Liang <zhaozongliang09@163.com>}{zhaozongliang09@163.com}}
\begin{abstract}
In this paper, we consider the non-linear general $p$-Laplacian equation $\Delta_{p,f}u+F(u)=0$ for a smooth function $F$ on smooth metric measure spaces. Assume that a Sobolev inequality holds true on $M$ and an integral Ricci curvature is small, we first prove a local gradient estimate for the equation. Then, as its applications, we prove several Liouville type results on manifolds with lower bounds of Ricci curvature. We also derive new local gradient estimates  provided that the integral Ricci curvature is small enough.


\end{abstract}
\maketitle \vskip0.2cm

\section{Introduction}
It is well-known that gradient estimates are an important tool in geometric analysis and have been used, among other things, to derive Liouville theorems and Harnack inequalities
for positive solutions to a variety of nonlinear equations on Riemannian manifolds.	Historically, the local Cheng-Yau gradient estimate asserts that 
that if $M$ is an $n$-dimensional complete Riemannian manifold with $Ric\geq -(n-1)\kappa$ for some $\kappa\geq0$ and $u: B(o, R)\subset M\to\RR$ harmonic and positive then there is a constant $c_n$ depending only on $n$ such that 
\begin{equation}\label{e11}
\sup\limits_{B(o,R/2)}\frac{|\nabla u|}{u}\leq c_n\frac{1+\ \sqrt[]{\kappa}R}{R}.
\end{equation}
Here $B(o,R)$ stands for the geodesic ball centered at a fixed point $o\in M$. When $\kappa=0$, this implies that a harmonic function with sublinear growth on a manifold with non-negative Ricci curvature is constant. Later, Cheng-Yau's gradient estimate has been extended and generalized by many mathematicians. To describe recent results, let us recall some notations. The triple $(M^n, g, e^{-f}d\mu)$ is called a smooth metric measure space if $(M, g)$ is a Riemannian manifold, $f$ is a smooth function on $M$ and $d\mu$ is the volume element induced by the metric $g$. On $M$, we consider the differential operator $\Delta_f$, which is called $f-$Laplacian and given by
$$ \Delta_f\ \cdot:=\Delta\cdot-\left\langle \nabla f,\ \nabla\cdot\right\rangle.  $$
It is symmetric with respect to the measure $e^{-f}d\mu$. That is,
$$ \intl_M\left\langle \nabla \varphi, \nabla\psi\right\rangle e^{-f}d\mu=-\intl_M(\Delta_f\varphi)\psi e^{-f}d\mu, $$
for any $\varphi, \psi\in C^\infty_0(M)$. Smooth metric measure spaces are also called manifolds with density. By $m$-dimensional Bakry-\'{E}mery Ricci tensor we mean
$$ Ric_f^m=Ric+Hessf-\frac{\nabla f\otimes\nabla f}{m-n}, $$
for $m > n$. The tensor ${\rm Ric}^n_f$
is only defined when $f$ is constant. In this case, this tensor is referred as $\infty-$Barky-\'{E}mery tensor 
$$ Ric_f=Ric+Hessf. $$  
In a variational point of view, the weighted $p$-Laplacian, $p>1$ is a natural generalization of $\Delta_f$ and is defined by
$$\Delta_{p,f}u:=e^f{\rm div}(e^{-f}|\nabla u|^{p-2}\nabla u)$$
for $u\in W^{1,p}_{loc}(M)$. In \cite{DD}, Dung and Dat considered $F(u)=\lambda u^{p-1}$ and studied gradient estimates for weighted $p$-eigenfunctions of $\Delta_{p,f}$.  If $F(u)=cu^\sigma$, \eqref{eq1} is a Lichnerowicz type equation. In \cite{zhaoyang}, the authors proved local gradient estimates for positive solutions to this equation, and as applications, they gave a corresponding Liouville property and Harnack inequality. Then, Wang \cite{W} estimated eigenvalues of the weighted $p$-Laplacian. Wang, Yang, and Chen \cite{WYC} established gradient estimates and entropy formulae for weighted $p$-heat equations. 
Later, Dung and Sung \cite{DS} investigated some Liouville properties for weighted $p$-harmonic $\ell$-forms on smooth metric measure spaces with Sobolev and Poincar\'{e} inequalities. For the general setting on metric spaces, recently in \cite{BBS}, the authors considered under which geometric conditions on the underlying metric measure space the finite-energy Liouville theorem holds for $p$-harmonic functions and quasiminimizers. For further discussion about this topic, we refer the reader to \cite{BBS, Hol, kot, Mo, Nak, Pigola, WZ, WYC} and the references therein. 
	
In another direction, gradient estimates have been successfully generalized on manifolds with integral Ricci curvature condition. Before stating results, let us fix some notation. for each $x\in M$, denote by $\rho(x)$ the smallest eigenvalue for the $m$-dimensional Bakry-\'{E}mery Ricci tensor ${\rm Ric}_f^m: T_xM\to T_xM$, and for any fixed number $K$, let 
$$({\rm Ric}_f^m)_{-}^K(x)=((n-1)K-\rho(x))_+=\max\{0, (n-1)K-\rho(x)\},$$
the amount of $m$-dimensional Bakry-\'{E}mery Ricci curvature lying below $(n-1)K$. Let
$$\|{\rm Ric}_{-}^K\|_{q,r}=\sup\limits_{x\in M}\left(\int\limits_{B(x,r)}\Big(({\rm Ric}_f^m)_{-}^K\Big)^qdvol\right)^\frac{1}{q}.$$
Then $\|\RicH\|_{q,r}$ measures the amount of $m$-dimensional Bakry-\'{E}mery Ricci curvature lying below a given bound, in this case, $(n-1)K$, in the $L^q$ sense. It is easy to see that $\|\RicH\|_{q,r}=0$ if and only if ${\rm Ric}_M\geq(n-1)K$. We also often work with the following scale invariant curvature quantity (with $K=0$)
$$k(x, q, r)=r^2\left(\oint\limits_{B(x,r)}\rho_{-}^q\right)^\frac{1}{q}, \quad\quad k(q,r)=\sup\limits_{x\in M}k(x,q,r),$$
where the notation 
$$\oint_{B(x,r)}(\cdot):=\frac{1}{|B(x,r)|}\int_{B(x,r)}(\cdot)$$ 
represents the average integral on $B(x,r)$ and $|B(x,r)|$ stands for the volume of $B(x,r)$. We should note that the integral curvature bound is a natural, and much weaker than lower bound Ricci curvature condition. It appears naturally in isospectral problems and geometric variation problems. For further discussion about the relationship between integral Ricci curvature condition and aspects of topology and geometry of manifolds, we refer the reader to \cite{BZ, Gal, PW1, PW2} and the references therein.   Recently, integral Ricci curvature conditions are used to give gradient estimates of possitive solutions to heat equations. In particular,  in \cite{Rose}, Rose investigated heat kernel upper bound on Riemannian manifolds with locally uniform Ricci curvature integral bounds. In \cite{RO}, Oliv\'{e} used the integral Ricci curvature to show a Li-Yau gradient estimate on a compact Riemannian manifolds with Neumann boundary condition. It is worth to mention that Li-Yau gradient estimates for linear heat equation on complete non-compact manifolds were obtained by Zhang and Zhu in \cite{ZZ, ZZ1}. Later, these results were generalizied by Wang in \cite{W2} to non-linear heat equation. Moreover, inspired by a method in  \cite{DD}, Wang derived a gradient estimate of Hamilton type for a non-linear heat equation in \cite{W1}.   

	Motivated by Liouville results for $p$-Laplacian obtained by Zhao and Yang in \cite{zhaoyang}, by Hou in \cite{Hou}, our aim is to give local gradient estimate for positive solutions of the following equation
	\begin{equation} \label{eq1}
	\Delta_{p,f}u+F(u)=0
	\end{equation}
	on non-compact smooth metric measure space. Throughout this paper, we assume that $F$ is a differentiable function, $ F(u)\geq 0$ when $ u\geq0 $. Let $h(v)=(p-1)^{p-1}e^{-v}F(e^{\frac{v}{p-1}})$, we assume further that $h'(v)\leq a:=a(p)$ for some constant $a\geq0$, where $a=0$ if $p\not=2$. We say that \textit{a weighted Sobolev inequality} holds true on $M$ if there exist positive constants $C_1, C_2, C_3$, depending only on $ m $, such that for every ball $ B_{0}(R)\subset M $, every function $ \phi \in C^{\infty}_{0}(B_{0}(R))$ we have
			\begin{equation}\label{sobolev1}
			\left( \int_{B_{0}(R)}|\phi|^{\frac{2m}{m-2}}e^{-f}d\mu\right) ^{\frac{m-2}{m}}\leq C_1e^{C_2(1+\sqrt{K}R)}V^{-C_3}\int_{B_{0}(R)}(R^{2}|\nabla \phi|^{2}+\phi^{2})e^{-f}d\mu,
			\end{equation}
			where $ V $ is volume of the geodesic ball $ B_{0}(R) $.	
			
	The main result of this paper can be stated as follows.
	\begin{theorem}\label{dl1}
		Let  $(M,g,d\mu)$ be a smooth metric space admitting a Sobolev inequality \eqref{sobolev1}. 		 Assume that $ u $ is a positive solution of \eqref{eq1} on the geodesic ball $ B_{0}(R)\subset M $ and $F(u)\geq0$ when $u\geq0, h'(v)\leq a=a(p)$ for some constant $a\geq0$, where $a=a(p)=0$ if $p\not=2$. For any $\eta>0$, there exists $b>0$ such that if $\|{\rm Ric}_{-}^K\|_{q,r}\leq \frac{1}{bR^2}$ and $k(q,1)\leq \frac{1}{b}$ then there exists a constant $ C_{p,m, V} $ which depends only on  $ p, m $ and $V$ and  such that
		\begin{equation} \label{kq1}
		\dfrac{|\nabla u|}{u}\leq C_{p,m, V}\dfrac{1+\sqrt{K}R}{R} +\eta,
		\end{equation}
		on the geodesic ball  $ B_{0}(\frac{R}{2}) $. 		However, if $\|{\rm Ric}_{-}^K\|_{q,r}=0$ then
				\begin{equation} \label{main}
		\dfrac{|\nabla u|}{u}\leq C_{p,m}\dfrac{1+\sqrt{K}R}{R} .
		\end{equation}
		on the geodesic ball  $ B_{0}(\frac{R}{2}) $, and $C(p,m)$ depends only on $p$ and $m$.
	\end{theorem}

Note that the condition $\|{\rm Ric}_{-}^K\|_{p,r}=0$ is equivalent to ${\rm Ric}_f^m\geq (n-1)K$. In this case, we do not need to require any bound for $k(q,1)$. Moreover, a Sobolev inequality also holds true on $M$.
\begin{lemma}[see \cite{DD, zhaoyang}]\label{sobolev22}
			Let $ (M,g,d\mu)$ be a smooth metric measure space of dimension $n$. Assume that $ Ric_{f}^{m}\geq-(m-1)K $ where $K$ is a non-negative constant, $ m>n\geq2 $. Then, there exists a constant $ C $, depending only on $ m $, such that for every ball $ B_{0}(R)\subset M $, every function $ \phi \in C^{\infty}_{0}(B_{0}(R))$ we have
			\begin{equation*}
			\left( \int_{B_{0}(R)}|\phi|^{\frac{2m}{m-2}}e^{-f}d\mu\right) ^{\frac{m-2}{m}}\leq e^{C(1+\sqrt{K}R)}V^{-\frac{2}{m}}\int_{B_{0}(R)}(R^{2}|\nabla \phi|^{2}+\phi^{2})e^{-f}d\mu,
			\end{equation*}
			where $ V $ is geodesic ball volume $ B_{0}(R) $.
		\end{lemma}
Now, combining Theorem \ref{dl1} and Lemma \ref{sobolev22}, we derive some applications of Theorem \ref{dl1}. Note that when $F(u)=cu^\sigma$, for some $c\geq 0$ and $0\leq\sigma\leq p-1, p>1$, we have that $h(v)=c(p-1)^{p-1}e^{\left(\frac{\sigma}{p-1}-1\right)v}$. Hence
$$h'(v)=c(p-1)^{p-1}\left(\frac{\sigma}{p-1}-1\right)e^{\left(\frac{\sigma}{p-1}-1\right)v}\leq0.$$
Therefore, for $K=0$, letting $R$ tend to infinity in \eqref{main}, we obtain the following corollary.
\begin{corollary} Let  $(M,g,d\mu)$   be a smooth metric space with  $ Ric_{f}^{m}\geq 0$. If $u$ is a positive solution to equation $\Delta_{p,f}u+cu^\sigma=0$ and is defined globally on the space then $u$ must be constant.
\end{corollary} 
This corollary is a refinement of a result by Zhao and Yang in \cite{zhaoyang}. In fact, in Theorem 1.1 in \cite{zhaoyang}, the authors proved that
$$\frac{|\nabla u|}{u}\leq \frac{(1+\ \sqrt[]{K}R)^{3/4}}{R}$$
if ${\rm Ric}_m^f\geq-(n-1)K, K\geq0$. However, the above estimate should be corrected as \eqref{main}.

 We now give a gradient estimate for the Allen-Cahn equation. This equation has its origin in the gradient theory of phase transitions \cite{allen}, and has attracted a lot
of attention in the last decades. It also has the intricate connection to the minimal surface theory, for example, see \cite{cho, PR, pino} and the references there in. Our result now can be stated as follows.
		\begin{corollary} \label{hq1}
		Let $ (M,g,d\mu) $ be a smooth metric measure space with $ Ric_{f}^{m}\geq-(m-1)K $, $ K $ is a non-negative constant. If $ u $ is a solution of the equation
		\begin{equation*}
		\Delta_{p,f}u+u(1-u^{2})=0, \quad  p\geq 2,
		\end{equation*}
		satisfying $ 0<u\leq 1 $ on the ball $ B_{0}(R)\subset M $ then
		\begin{equation*} 
		\dfrac{|\nabla u|}{u}\leq C_{p,m}\dfrac{1+\sqrt{K}R}{R} 
		\end{equation*}
		on the ball  $ B_{0}(\frac{R}{2}) $, where $ C_{p,m} $ is a constant depending only on $ p $ and $ m $.
		In particular, when $ K=0 $, if $0<u\leq1$ in $M$, then $ u\equiv 1 $ on $ M $.
	\end{corollary}
	Note that for $p=2$, this kind of Liouville type theorem was verified by S. B. Hou in \cite{Hou}. This corollary can be considered as a generalization of those in \cite{Hou} in the non-linear setting. It is also worth to emphasize that the above gradient is new, even for $p=2$.
	
	The second application is a new gradient estimate for the Fisher-KPP equation.
	\begin{corollary} \label{hq2}
		Let $ (M,g,d\mu) $ be a smooth metric space with $ Ric_{f}^{m}\geq-(m-1)K $, constant $ K\geq 0 $. If $ u $ is a positive solution of the equation
		\begin{equation*}
		\Delta_{p,f}u+cu(1-u)=0, \quad p\geq 2, c>0,
		\end{equation*}
		on the geodesic ball $ B_{0}(R)\subset M $, $u \le 1$ in $M$ then
		\begin{equation*} 
		\dfrac{|\nabla u|}{u}\leq C_{p,m}\dfrac{1+\sqrt{K}R}{R} 
		\end{equation*}
		on the geodesic ball  $ B_{0}(\frac{R}{2}) $, with $ C_{p,m} $ only depends on $ p $ and $ m $. When $ K=0 $ then $ u\equiv 1 $ on $ M $.
	\end{corollary}
	The equation in Corollary \ref{hq2} was proposed by Fisher in 1937 to describe the propagation of an evolutionarily advantageous gene in a population \cite{Fisher}, and was also independently described in a seminal paper by Kolmogorov, Petrovskii, and Piskunov in the same year \cite{KKP}. In \cite{Cao}, the authors derived several differential Harnack estimates for positive solutions of Fisher’s equation. 
	
	The third application is the below Liouville result.
	\begin{corollary} \label{hq31}
		Let $ (M,g,d\mu) $ be a smooth metric space with $ Ric_{f}^{m}\geq-(m-1)K $, $ K\geq 0 $. If $ u\geq1 $ is a solution of the equation
		\begin{equation}\label{soliton}
		\Delta_{f}u+au\log u=0, \quad a\geq0,
		\end{equation}
		on the geodesic ball $ B_{0}(R)\subset M $, then
		\begin{equation*} 
		\dfrac{|\nabla u|}{u}\leq C_{p,m}\dfrac{1+\sqrt{K}R}{R} 
		\end{equation*}
		on the geodesic ball  $ B_{0}(\frac{R}{2}) $, with $ C_{p,m} $ only depends on $ p $ and $ m $. When $ K=0 $ and $u\geq 1$ in $M$ then $ u\equiv 1 $ in $ M $.
	\end{corollary}  
Note that equation \eqref{soliton} originated from gradient Ricci solitons. We refer the reader to \cite{DD1} for further explaination. It is worth to mention that in \cite{DK, wu}, the authors showed that there does not exist positive solution satisfying $0<u\leq c< 1$ for some $c\in\mathbb{R}$. 

	The paper has three sections. Beside this section, we prove Theorem \ref{dl1} in the Section 2. As its applications, we derive proof of corollaries in the Section 3 and point our some local gradient estimate under integral Ricci curvature condition. 
	
	\noindent
\textbf{Acknowledgment: } This work was initiated during a visit of the second author to HongKong Univeristy of Science and Technology (HKUST) and Vietnam Institute for Advanced Study in Mathematics (VIASM). He would like to thank Tianling Jin (HKUST) and VIASM for their kind invitation and support. 
\section{Gradient estimate with a Sobolev inequality and integral Ricci bounds}
\setcounter{equation}{0}	
Since the equation	 \eqref{eq1} can be either degenerate or singular in the set $\{\nabla u=0\}$, the elliptic regular theory may not be applied. It is well known that the best regular properties of the solution of this kind of equations is $C^{1, \alpha}$, for some $0<\alpha<1$. As in \cite{zhaoyang} (see also \cite{Hard, WZ}),  using an $\varepsilon$-regularization technique by replacing the linearized operator $\mathcal{L}_f$ (see below definition) with its approximate, we can assume that $u$ is smooth. Therefore, in order to avoid tedious presentation, throughout this paper, for simplicity, we assume that $ u $ is a positive $\mathcal{C}^2$-solution of \eqref{eq1}. Put
	\begin{equation*}
	v=(p-1)\mathrm{log}u, \quad w=|\nabla v|^{2}.
	\end{equation*}
To prove Theorem \ref{dl1}, we need to use the following operator.
	\begin{define}[\cite{wangli, WZ}] Linearization operator of the weighted $ p $-Laplacian corresponding with $ u\in \mathcal{C}^{2}(M)$ such that $ \nabla u \neq 0 $ is defined as follows
		\begin{equation*}
		\mathcal{L}_{f}(\psi)=e^{f}\mathrm{div}(e^{-f}|\nabla u|^{p-2}A(\nabla\psi)),
		\end{equation*}
		where $ \psi $ is a smooth function on $ M $ and $ A $ is a tensor defined by
		\begin{equation*}
		A=\mathrm{Id}+(p-2)\dfrac{\nabla u \otimes \nabla u}{|\nabla u|^{2}}.
		\end{equation*}
	\end{define}  
	
	\begin{lemma}\label{bd3} (\cite{wangli}) Let $ (M, g, d\mu) $ be a smooth metric space and function $ u\in \mathcal{C}^{3}(M) $. Then, if  $ |\nabla u|\neq 0 $, then
		\begin{equation*}
		\mathcal{L}_{f}(|\nabla u|^{p})=p|\nabla u|^{2p-4}\left( \left| \mathrm{Hess } u \right| _{A}^{2}+\mathrm{Ric}_{f}(\nabla u, \nabla u)\right) +p|\nabla u|^{p-2}\left\langle\nabla u,  \nabla \Delta_{p,f} u \right\rangle  .
		\end{equation*}
		where $ |\mathrm{Hess} u|_{A}^{2}=A^{ik}A^{jl}u_{ij}u_{kl} $ and $ A $ are defined as above.
	\end{lemma}

	To estimate the Hessian term, we need the following lemma.
	
	\begin{lemma} \label{bdhessA}
		For $ v=(p-1)\mathrm{log}u, w=|\nabla v|^2 $, and $ \alpha=\min \left\lbrace 2(p-1), \frac{m(p-1)^{2}}{m-1} \right\rbrace  $, let 
		$$ h(v)=(p-1)^{p-1}e^{-v}F(e^{\frac{v}{p-1}}),$$
then we have
		\begin{equation*}
		\begin{aligned}
		|\mathrm{Hess} v|_{A}^{2}  \geq & \dfrac{\alpha}{4}\dfrac{|\nabla w|^{2}}{w} + \dfrac{w^{2}}{m-1}(1+hw^{\frac{-p}{2}})^{2} \\
		&+ \dfrac{p-1}{m-1}(1+hw^{\frac{-p}{2}})\left\langle \nabla v, \nabla w \right\rangle -  \dfrac{(f_{1}v_{1})^{2}}{m-n}.
		\end{aligned}
		\end{equation*}
	\end{lemma}
	
	\begin{proof}
		Substituting $ v $ into the equation $ \eqref{eq1} $, we obtain
		\begin{equation*}
		\begin{aligned}
		0=\Delta_{p,f}u+F(u) &=e^{f}\mathrm{div}(e^{-f}|\nabla e^{\frac{v}{p-1}} |^{p-2} \nabla e^{\frac{v}{p-1}})+F(e^{\frac{v}{p-1}}) \\
		&=(p-1)^{1-p}e^{v}(|\nabla v|^{p}+\Delta_{p,f}v)+F(e^{\frac{v}{p-1}}).
		\end{aligned}
		\end{equation*}
		Hence 
		\begin{align}  \label{eq2}
		\Delta_{p,f}v&=-(p-1)^{p-1}e^{-v}F(e^{\frac{v}{p-1}})-|\nabla v|^{p}\notag\\
&=-h(v)-w^\frac{p}{2}.
		\end{align}
		By the definition of the weighted $ p $-Laplacian, this implies
	
	\begin{equation}\label{eq3}
	w^{\frac{p-2}{2}}\Delta_{f}v+\frac{p-2}{2}\left\langle \nabla w,  \nabla v \right\rangle w^{\frac{p-4}{2}}
		=-h-w^{\frac{p}{2}}.
		\end{equation}
		We need to estimate $ |\mathrm{Hess} v|_{A}^{2} $ at points where $ w>0 $. Choose a local orthogonal basis $ \left\lbrace e_i \right\rbrace ^{n}_{i=1} $ near a given point such that $ \nabla v=|\nabla v|e_{1} $. We use $\nabla_{e_i}w=w_i, i=\overline{1,n}$ then $ w=v_{1}^{2}, w_{1}=2v_{i1}v_{i}=2v_{11}v_{1} $, when $ j\geq 2, w_{j}=2v_{j1}v_{1} $. Therefore, $ 2v_{j1}=\dfrac{w_{j}}{w^{\frac{1}{2}}} $, $ \left\langle \nabla f, \nabla v\right\rangle =f_{1}v_{1} $. Hence \eqref{eq3} leads to
		\begin{equation} \label{eq4}
		\begin{aligned} 
		\sum_{j=2}^{n}v_{jj}&=-hw^{1-\frac{p}{2}}-(\frac{p}{2}-1)\dfrac{w_{1}v_{1}}{w}-v_{11}+f_{1}v_{1}-w\\
		&=-hw^{1-\frac{p}{2}}-(p-1)v_{11}+f_{1}v_{1}-w.
		\end{aligned}
		\end{equation}
		From the definition of matrix $ A $, we have
		\begin{equation*}
		|\mathrm{Hess} v|_{A}^{2}=|\mathrm{Hess} v|^{2} + \dfrac{(p-2)^{2}}{4w^{2}}\left\langle \nabla v, \nabla w \right\rangle ^{2} + \dfrac{p-2}{2w}|\nabla w|^{2}.
		\end{equation*}
		Using the Cauchy-Schwarz inequality, we obtain
		\begin{align*}
		|\mathrm{Hess} v|_{A}^{2}
		=& \sum_{i,k=1}^{n}v_{ij} + (p-2)^{2}v_{11}^{2} + 2(p-2) \sum_{k=1}^{n}v_{1k}^{2}\\
		=& (p-1)^{2}v_{11}^{2} + 2(p-1)\sum_{k=2}^{n}v_{1k}^{2}+\sum_{i,k=2}^{n}v_{ik}^{2}\\
		&\geq (p-1)^{2}v_{11}^{2} + 2(p-1)\sum_{k=2}^{n}v_{1k}^{2}+ \dfrac{1}{n-1}\left( \sum_{j=2}^{n}v_{jj}\right)^{2}. 
		\end{align*}
		Substituting \eqref{eq4} into the above inequality, we have
		\begin{align*}
		|\mathrm{Hess} v|_{A}^{2}\geq & (p-1)^{2}v_{11}^{2} + 2(p-1)\sum_{k=2}^{n}v_{1k}^{2}\\
		& + \dfrac{1}{n-1}(-hw^{1-\frac{p}{2}}-(p-1)v_{11}+f_{1}v_{1}-w)^{2}.
		\end{align*}
		Using inequality $ (x-y)^{2}\geq \frac{x^{2}}{1+\delta}-\frac{y^{2}}{\delta} $ for $x=hw^{1-\frac{p}{2}}+w+(p-1)v_{11} $, $y=f_{1}v_{11} $, we have
		\begin{equation*}
		\begin{aligned}
		&\dfrac{1}{n-1}(-hw^{1-\frac{p}{2}}-(p-1)v_{11}+f_{1}v_{1}-w)^{2}\\
		&\geq \dfrac{(hw^{1-\frac{p}{2}}+w)^{2}+2(p-1)v_{11}(hw^{1-\frac{p}{2}}+w)}{m-1}
		+\dfrac{(p-1)^{2}}{m-1}v_{11}^{2}-\dfrac{(f_{1}v_{1})^{2}}{m-n}
		\end{aligned}
		\end{equation*}
		Denote $ \alpha=\min\left\lbrace 2(p-1), \frac{m(p-1)^{2}}{m-1}\right\rbrace  $, we obtain
		\begin{equation*}
		\begin{aligned}
		|\mathrm{Hess} v|_{A}^{2}\geq &\alpha \sum_{k=1}^{n}v_{1k}^{2}+\dfrac{1}{m-1}(hw^{1-\frac{p}{2}}+ w)^{2}\\
		&+ \dfrac{2(p-1)v_{11}}{m-1}(hw^{1-\frac{p}{2}}+w)-\dfrac{(f_{1}v_{1})^{2}}{m-n}.
		\end{aligned}
		\end{equation*}
	 Observe that
		\begin{equation*}
		2wv_{11}=\left\langle\nabla v, \nabla w \right\rangle, \sum_{j=1}^{n}v_{1j}^{2}=\dfrac{1}{4}\dfrac{|\nabla w|^{2}}{w}. 
		\end{equation*}
		Substituting these identities into the above inequality, we have
		\begin{equation*}
		\begin{aligned}
		|\mathrm{Hess} v|_{A}^{2}  \geq & \dfrac{\alpha}{4}\dfrac{|\nabla w|^{2}}{w} + \dfrac{w^{2}}{m-1}(1+hw^{\frac{-p}{2}})^{2} \\
		&+ \dfrac{p-1}{m-1}(1+hw^{\frac{-p}{2}})\left\langle \nabla v, \nabla w \right\rangle -  \dfrac{(f_{1}v_{1})^{2}}{m-n}.
		\end{aligned}
		\end{equation*}
		The proof is complete.	
	\end{proof}
	
	Now we will estimate $ \mathcal{L}_{f}(Q) $, for $ Q=|\nabla v|^{p} $. From  \eqref{eq2}, we obtain
$$\nabla\Delta_{p,f}v=-h'(v)\nabla v-\nabla (|\nabla v|^p).$$
		Combining this identity with Lemma \ref{bd3},  we have
	\begin{equation*}
	\begin{aligned}
	\mathcal{L}_{f}(Q)=pw^{p-2}\left( \left| \mathrm{Hess} v \right| _{A}^{2}+\mathrm{Ric}_{f}(\nabla v, \nabla v)\right)-pw^{\frac{p-2}{2}}\left\langle \nabla v, \nabla Q \right\rangle-ph'(v)w^{\frac{p}{2}-1}|\nabla v|^{2}.
	\end{aligned}
	\end{equation*}
	Using Lemma \ref{bdhessA} and the above result equation with note that the function $h$ satisfying $h'(v)\leq a $, we infer 
	\begin{equation*}
	\begin{aligned}
	\mathcal{L}_{f}(Q)=& \mathcal{L}_{f}(w^{\frac{p}{2}})\\
	\geq & pw^{p-2}\left(  \dfrac{\alpha}{4}\dfrac{|\nabla w|^{2}}{w}+ \dfrac{1}{m-1}w^{2} ( 1+hw^{\frac{-p}{2}})^{2}+\dfrac{p-1}{m-1}(1+hw^{\frac{-p}{2}})\left\langle \nabla v, \nabla w \right\rangle\right)  \\
	&+pw^{p-2}\mathrm{Ric}_{f}^{m}(\nabla v, \nabla v)-pw^{\frac{p-2}{2}}\left\langle \nabla v, \nabla w^{\frac{p}{2}} \right\rangle-paw^\frac{p}{2}.
	\end{aligned}
	\end{equation*}
This inequality can be written as follows.
	\begin{equation*}
	\begin{aligned}
	\mathcal{L}_{f}(Q)
	\geq &  \dfrac{\alpha p}{4}w^{p-3}|\nabla w|^{2} + \dfrac{p}{m-1}w^{p}(1+hw^{\frac{-p}{2}})^{2}\\
	&+\left( \dfrac{p(p-1)}{m-1}(1+hw^{\frac{-p}{2}})-\dfrac{p^{2}}{2}\right) w^{p-2}\left\langle \nabla v, \nabla w \right\rangle+p\mathrm{Ric}_{f}^{m}(\nabla v, \nabla v)w^{p-2}-apw^\frac{p}{2}.
	\end{aligned}
	\end{equation*}
	Note that the above inequality holds when $ w >0$. In order to pass through $\{w=0\}$, we put $\mathcal{S}=\left\lbrace x\in M: w(x)=0\right\rbrace  $. Then we need to have the following lemma. In the rest of this section, integration is taken with respect to $e^{-f}d\mu$. Moreover, we skip $e^{-f}d\mu$ for simplicity of notations.
	
	\begin{lemma} \label{tichphan} Let $ \psi $ be a non-negative Lipschitz function with  compact support on $ M \setminus \mathcal{S}$, then
		\begin{equation*}
		\begin{aligned}
		&\int_{\Omega} \mathcal{L}_{f}(Q)\psi\\
		&\geq  \int_{\Omega}\Bigg( \dfrac{\alpha p}{4}w^{p-3}|\nabla w|^{2} + \dfrac{p}{m-1}w^{p}(1+hw^{\frac{-p}{2}})^{2}-apw^\frac{p}{2}\\
		&+\left( \dfrac{p(p-1)}{m-1}(1+hw^{\frac{-p}{2}})-\dfrac{p^{2}}{2}\right) w^{p-2}\left\langle \nabla v, \nabla w \right\rangle+p{\rm Ric}_f^m(\nabla v, \nabla v)w^{p-2}\Bigg)\psi.
		\end{aligned}
		\end{equation*}
	\end{lemma}
	
	\begin{proof} 
We first use integration by parts to obtain
		\begin{equation*}
		\begin{aligned}
		\int_{\Omega} \mathcal{L}_{f}(Q)\psi d\mu=&\int_{\Omega}e^{f}\mathrm{div}\left(e^{-f}|\nabla v|^{p-2} A(\nabla Q) \right)\psi e^{-f}dv\\
				=&-\int_{\Omega}|\nabla v|^{p-2}\left\langle A(\nabla Q), \nabla \psi \right\rangle d\mu.
		\end{aligned}
		\end{equation*}
		Since \begin{equation*}
		A=Id+(p-2)\dfrac{\nabla v\otimes \nabla v}{|\nabla v|^{2}}, \qquad \nabla Q=\dfrac{p}{2}w^{\frac{p-2}{2}}\nabla w
		\end{equation*}
		we have
		\begin{equation*}
		\begin{aligned}
		A(\nabla Q)=\dfrac{p}{2}w^{\frac{p-2}{2}}\nabla w+\dfrac{1}{2}p(p-2)w^{\frac{p-4}{2}}\left\langle \nabla v,\nabla w\right\rangle \nabla v.\\
		\end{aligned}
		\end{equation*}
		Combining these identities with Lemma \ref{tichphan}, we conclude that
		\begin{equation*}
		\begin{aligned}
		\int_{\Omega} \mathcal{L}_{f}(Q)\psi=-&\int_{\Omega}\left\langle \dfrac{1}{2}w^{p-2}\nabla w + \dfrac{1}{2}(p-2)w^{p-3}\left\langle \nabla v, \nabla w \right\rangle \nabla v, \nabla \psi \right\rangle \\
		\geq &  \int_{\Omega}\Bigg( \dfrac{\alpha p}{4}w^{p-3}|\nabla w|^{2} + \dfrac{p}{m-1}w^{p}(1+hw^{\frac{-p}{2}})^{2}-apw^\frac{p}{2}\\
		&+\left( \dfrac{p(p-1)}{m-1}(1+hw^{\frac{-p}{2}})-\dfrac{p^{2}}{2}\right) w^{p-2}\left\langle \nabla v, \nabla w \right\rangle+p{\rm Ric}_f^m(\nabla v, \nabla v)w^{p-2}\Bigg)\psi.
		\end{aligned}
		\end{equation*}
The proof is complete.
	\end{proof}
	We now assume that $M$ satisfies a weighted Sobolev inequalitiy. This means there exist positive constants $C_1, C_2, C_3$, depending only on $ m $, such that for every ball $ B_{0}(R)\subset M $, every function $ \phi \in C^{\infty}_{0}(B_{0}(R))$ we have
			\begin{equation}\label{sobolev}
			\left( \int_{B_{0}(R)}|\phi|^{\frac{2m}{m-2}}\right) ^{\frac{m-2}{m}}\leq C_1e^{C_2(1+\sqrt{K}R)}V^{-C_3}\int_{B_{0}(R)}(R^{2}|\nabla \phi|^{2}+\phi^{2})d\mu,
			\end{equation}
			where $ V $ is geodesic ball volume $ B_{0}(R) $.
	Using this Sobolev inequality and Lemma \ref{tichphan}, we can prove the following result which is an important ingredient in the proof of Theorem \ref{dl1}.	
\begin{lemma}[$ \mathbf{L^{q}}$-norm estimate]\label{bd2}  With the same assumption as in Theorem \ref{dl1}, if $ b_{0}>0 $ large enough, then there exists $ d_1(p,m)>0 $ such that
		\begin{equation*}
		\parallel w \parallel_{L^{(b_{0}+p-1)\frac{m}{m-2}}(B_{0}(\frac{3}{4}R))}
		\leq d_{1}\dfrac{b_{0}^2}{R^{2}}V^{\frac{m-2}{m(b_{0}+p-1)}}.
		\end{equation*}
	\end{lemma} 
\begin{proof}			
We choose  $\psi= w_{\epsilon}^{b}\eta^{2} $, where $ \epsilon > 0 $, $ \eta \in \mathcal{C}_{0}^{\infty}(B_{0}(R)) $ and $ w_{\epsilon}=(w-\epsilon)^{+}$.  Plugging $\psi$ into Lemma \ref{tichphan}, we obtain an inequality which is the same as the equation (2.3) in \cite{zhaoyang}. Therefore, we can use the same arguments as in \cite{zhaoyang}, after letting $\epsilon$ tend to zero and doing some direct computations, we obtain (see the conclusion before Lemma 2.2 in \cite{zhaoyang})
\begin{align}
\int_{B_o(R)}&\left|\nabla\left(w^{\frac{p+b-1}{2}}\eta\right)\right|^2+bd_1\int_{B_o(R)}w^{p+b}\eta^2\notag\\
\leq& a_0\int_{B_o(R)}w^{p+b-1}|\nabla\eta|^2-bd_2\int_{B_o(R)}{\rm Ric}_f^m(\nabla v, \nabla v)w^{p+b-2}\eta^2\notag\\
&+bad_3\int_{B_o(R)}w^{\frac{p}{2}+b}\eta^2.\label{ric1}
\end{align}
for some positive constants $a_0, d_1, d_2, d_3\in\mathbb{R}^+$ and $b\cong\frac{1}{m-1}$. From now on, $a_0, a_{1}, a_{2},... $ and $ d_{1},d_{2},... $ are coefficients depending only on $ p $ and $ m $. We now estimate the Ricci term. By H\"{o}lder inequality, we have 
\begin{align}
\int_{B_o(R)}&{\rm Ric}_f^m(\nabla v, \nabla v)w^{p+b-2}\eta^2\notag\\
\geq&(n-1)K\int_{B_o(R)}w^{p+b-1}\eta^2-\int_{B_o(R)}|({\rm Ric}_{f}^{m})_-^K|w^{p+b-1}\eta^2\notag\\
\geq&(n-1)K\int_{B_o(R)}w^{p+b-1}\eta^2-\|({\rm Ric}_{f}^{m})_-^K\|^q\left(\int_{B_o(R)}(w^{p+b-1}\eta^2)^{q/(q-1)}\right)^\frac{q-1}{q}.\label{ric2}
\end{align}
Now, we use a technique in \cite{Daietal} to process as follows. We put $\alpha=\frac{2q-n}{2(q-1)}$ and $\theta=\frac{m}{m-2}$ then 
$$\alpha+(1-\alpha)\theta=\frac{q}{q-1}.$$
Using H\"{o}lder inequality, for any $\varepsilon>0$, we have 
$$\begin{aligned}
\left(\int_{B_o(R)}(w^{p+b-1}\eta^2)^{q/(q-1)}\right)^\frac{q-1}{q}\leq&\left(\int_{B_o(R)}w^{p+b-1}\eta^2\right)^{\frac{q-1}{q}\alpha}\cdot\left(\int_{B_o(R)}(w^{p+b-1}\eta^2)^\theta\right)^{(1-\alpha)\frac{q-1}{q}}\\
\leq&\varepsilon\left(\int_{B_o(R)}(w^{p+b-1}\eta^2)^\theta\right)^{\frac{1}{\theta}}+\varepsilon^{-\frac{(1-\alpha)\theta}{\alpha}}\cdot\left(\int_{B_o(R)}(w^{p+b-1}\eta^2)\right),
\end{aligned}$$
where in the last inequality, we used Young's inequality
$$xy\leq \varepsilon x^\gamma+\varepsilon^{-\frac{\gamma^*}{\gamma}}y^{\gamma^*},\ \forall x, y\geq0, \gamma>1, \frac{1}{\gamma}+\frac{1}{\gamma^*}=1.$$
By \eqref{sobolev}, this implies
\begin{align}
\left(\int_{B_o(R)}(w^{p+b-1}\eta^2)^{q/(q-1)}\right)^\frac{q-1}{q}
\leq&\varepsilon C_1e^{C_2(1+\ \sqrt[]{K}R)}V^{-C_3}\int_{B_o(R)}\left(R^2\left|\nabla\left(w^{\frac{p+b-1}{2}}\eta\right)\right|^2+w^{p+b-1}\eta^2\right)\notag\\
&+\varepsilon^{-\frac{(1-\alpha)\theta}{\alpha}}\cdot\left(\int_{B_o(R)}(w^{p+b-1}\eta^2)\right).\label{ric3}
\end{align}
Combining \eqref{ric1}-\eqref{ric3} and choose $\varepsilon=\frac{1}{2bd_1(C_1e^{C_2(1+\ \sqrt[]{K}R)}V^{-C_3}R^2)\|({\rm Ric}_f^m)_-^K\|}$, we conclude that
\begin{align}
\int_{B_o(R)}&\left|\nabla\left(w^{\frac{p+b-1}{2}}\eta\right)\right|^2+bd_1\int_{B_o(R)}w^{p+b}\eta^2\notag\\
\leq& a_0\int_{B_o(R)}w^{p+b-1}|\nabla\eta|^2-(n-1)bd_2K\int_{B_o(R)}w^{p+b-1}\eta^2\notag\\
&+bad_3\int_{B_o(R)}w^{\frac{p}{2}+b}\eta^2+d_4(be^{C_2(1+\sqrt[]{K}R)}V^{-C_3}R^2\|({\rm Ric}_m^f)_-^K\|)^\frac{n}{2q-n}\int_{B_o(R)}w^{p+b-1}\eta^2.\notag
\end{align}
Since
\begin{equation}\label{condition}
a=\begin{cases}
0,\quad&\text{if }p\not=2\\
\geq0&\text{if }p=2
\end{cases}, \quad \|{\rm Ric}_-^K\|_{q,r}\leq \frac{c}{be^{C_2(1+\sqrt[]{K}R)}V^{-C_3}R^2}, 
\end{equation}
and $\frac{p}{2}+b=p+b-1$ when $p=2$, the above inequality implies 
\begin{align}
\int_{B_o(R)}&\left|\nabla\left(w^{\frac{p+b-1}{2}}\eta\right)\right|^2+bd_1\int_{B_o(R)}w^{p+b}\eta^2\notag\\
\leq& a_1\int_{B_o(R)}w^{p+b-1}|\nabla\eta|^2+Kbd_3\int_{B_o(R)}w^{p+b-1}\eta^2
\end{align}
Combining this inequality with Sobolev inequality \eqref{sobolev}, we obtain
			\begin{equation} \label{eq12}
			\begin{aligned}
			&\left( \int_{B_{0}(R)}(w^{\frac{p+b-1}{2}}\eta)^{\frac{2m}{m-2}}\right) ^{\frac{m-2}{m}} + bd_{1}R^{2}e^{c_{2}b_{0}}V^{-\frac{2}{m}}\int_{B_{0}(R)}w^{p+b}\eta^{2}\\
			 \leq& d_{2}R^{2}e^{c_{2}b_{0}}V^{-\frac{2}{m}}\int_{B_{0}(R)} w^{p+b-1}|\nabla \eta|^{2}\\
			&+Kbd_{3}R^{2}e^{c_{2}b_{0}}V^{-\frac{2}{m}}\int_{B_{0}(R)} p(m-1)w^{p+b-1}\eta^{2}+e^{c_{2}b_{0}}V^{-\frac{2}{m}}\int_{B_{0}(R)}w^{p+b-1}\eta^{2}\\
			\leq& d_{2}R^{2}e^{c_{2}b_{0}}V^{-\frac{2}{m}}\int_{B_{0}(R)} w^{p+b-1}|\nabla \eta|^{2}+a_1b_{0}b^{2}e^{c_{2}b_{0}}V^{-\frac{2}{m}}\int_{B_{0}(R)}w^{p+b-1}\eta^{2}.
			\end{aligned}
			\end{equation}
			where $ b_{0}=c_{1}(m,p)(1+\sqrt{K}R) $ with $ c_{1} $ large enough.	Choose $ \eta_{1}\in C^{\infty}_{0}(\Omega) $ satisfying $ 0\leq\eta_{1}\leq1$, $ \eta_{1}\equiv 1 $ on $ B_{0}(\frac{3}{4}R) $, $ |\nabla \eta_{1}| \leq \dfrac{C_{1}}{R} $ and put $ \eta=\eta_{1}^{p+b} $. Then
			\begin{align*}
		d_{2}R^{2}\int_{B_{0}(R)}w^{p+b-1}|\nabla \eta|^{2}
			&\leq a_2b^{2}\int_{B_{0}(R)}w^{p+b-1}\eta^\frac{{2(p+b-1)}}{p+b}\\
			&\leq a_2b^{2}\left( \int_{B_{0}(R)}w^{p+b-1}\eta^2\right) ^{\frac{{p+b-1}}{p+b}}V^{\frac{1}{p+b}}\\
			&\leq \dfrac{bd_{1}}{2}R^{2}\int_{B_{0}(R)}w^{p+b-1}\eta^2+\left( \dfrac{a_3}{R^{2}}\right)^{p+b-1} b^{p+b+1}V,
			\end{align*}
			where we used the H\"{o}lder inequality and the Young inequality in the last two inequalities. Let $b=b_0$, this implies
			\begin{align}
			d_{2}R^{2}e^{c_{2}b_{0}}V^{\frac{-2}{m}}\int_{B_{0}(R)}w^{p+b-1}|\nabla \eta|^{2}\leq& \dfrac{bd_{1}}{2}R^{2}e^{c_{2}b_{0}}V^{\frac{-2}{m}}\int_{B_{0}(R)}w^{p+b}\eta^{2}\notag\\
			&+\left( \dfrac{a_3}{R^{2}}\right)^{p+b-1}b^{p+b+1}e^{c_{2}b_{0}}V^{1-\frac{2}{m}}\label{note1}\\
			\leq& \dfrac{bd_{1}}{2}R^{2}e^{c_{2}b_{0}}V^{-\frac{2}{m}}\int_{B_{0}(R)}w^{p+b}\eta^{2}\notag\\
			&+\left(\dfrac{a_4b_0^2}{R^{2}}\right)^{p+b_{0}-1}V^{1-\frac{2}{m}}\label{eq13}
			\end{align} 
			We estimate the second term of the right hand side of \eqref{eq12}. We see that $ a_1b^{2}_{0}bw^{p+b-1}<\frac{1}{2}bd_{1}R^{2}w^{p+b} $ when $ w>a_5b^{2}_{0}R^{-2} $. Therefore, to estimate the term, we devide $ B_{0}(R) $ into two domains $ B_{1} $ and $ B_{2} $ such that
			\begin{equation*}
			w\mid _{B_{1}}>a_5b^{2}_{0}R^{-2}; \quad 	w\mid _{B_{2}}\leq a_5b^{2}_{0}R^{-2}.
			\end{equation*}
			Since $ 0\leq \eta \leq 1 $, we have
			\begin{equation*}
			\begin{aligned}
			&a_1b^{2}_{0}be^{c_{2}b_{0}}V^{-\frac{2}{m}}\int_{B_{0}(R)}w^{p+b-1}\eta^{2}\\
			&=a_1b^{2}_{0}be^{c_{2}b_{0}}V^{-\frac{2}{m}}\left( \int_{B_{1}}w^{p+b-1}\eta^{2}+\int_{B_{2}}w^{p+b-1}\eta^{2}\right)\\
			&\leq \dfrac{1}{2}bd_{1}R^{2}e^{c_{2}b_{0}}V^{-\frac{2}{m}} \int_{B_{1}}w^{p+b}\eta^{2}+ a_1b_{0}^{2}be^{c_{2}b_{0}}V^{-\frac{2}{m}}\int_{B_{2}}w^{p+b-1} \\
			&\leq \dfrac{1}{2}bd_{1}R^{2}e^{c_{2}b_{0}}V^{-\frac{2}{m}} \int_{B_{1}}w^{p+b}\eta^{2}+ a_1b_{0}^{2}be^{c_{2}b_{0}}V^{-\frac{2}{m}}\int_{B_{2}}\left( \dfrac{a_5b_{0}^{2}}{R^{2}}\right)^{p+b-1} \\
			&\leq \dfrac{1}{2}bd_{1}R^{2}e^{c_{2}b_{0}}V^{-\frac{2}{m}} \int_{B_{0}(R)}w^{p+b}\eta^{2}+ a_1b_{0}^{2}be^{c_{2}b_{0}}V^{-\frac{2}{m}}\int_{B_{0}(R)}\left( \dfrac{a_5b_{0}^{2}}{R^{2}}\right)^{p+b-1} \\
			\end{aligned}
			\end{equation*}
			Observe that 
			\begin{equation*}
			\begin{aligned}
			a_1b_{0}^{2}be^{c_{2}b_{0}}V^{-\frac{2}{m}}\int_{B_{0}(R)}\left( \dfrac{a_5b_{0}^{2}}{R^{2}}\right)^{p+b-1}\leq &
			a_1b_{0}^{3}e^{c_{2}b_{0}}\left( \dfrac{a_5b_{0}^{2}}{R^{2}}\right)^{p+b_{0}-1}V^{1-\frac{2}{m}}
			\\
			\leq & \left( \dfrac{a_6b_{0}^{2}}{R^{2}}\right)^{p+b_{0}-1}V^{1-\frac{2}{m}} .\\
			\end{aligned}
			\end{equation*}
Combining this observation and the previous inequality, we infer 
						\begin{equation} \label{eq14}
			\begin{aligned}
			&a_1b^{2}_{0}be^{c_{2}b_{0}}V^{-\frac{2}{m}}\int_{B_{0}(R)}w^{p+b-1}\eta^{2}\\
			&\leq \dfrac{1}{2}bd_{1}R^{2}e^{c_{2}b_{0}}V^{-\frac{2}{m}} \int_{B_{0}(R)}w^{p+b}\eta^{2}+\left( \dfrac{a_6b_{0}^{2}}{R^{2}}\right)^{p+b_{0}-1}V^{1-\frac{2}{m}} .\\
			\end{aligned}
			\end{equation}
			Substituting \eqref{eq13}, \eqref{eq14} into \eqref{eq12}, we obtain
			\begin{equation*}
			\left( \int_{B_{0}(R)}(w^{\frac{p+b-1}{2}}\eta)^{\frac{2m}{m-2}}\right) ^{\frac{m-2}{m}} \leq \left(\dfrac{a_7}{R^{2}}b_{0}^{2}\right)^{p+b_{0}-1}V^{1-\frac{2}{m}}.
			\end{equation*}
As a consequence, this implies
			\begin{equation*}
			\parallel w \parallel_{L^{(b_{0}+p-1)\frac{m}{m-2}}(B_{0}(\frac{3}{4}R))}
			\leq d_{4}\dfrac{b_{0}^2}{R^{2}}V^{\frac{m-2}{m(b_{0}+p-1)}}.
			\end{equation*}
			We are done.
		\end{proof}
				Next we will prove  Theorem \ref{dl1}.
		
		\begin{proof}
Observe that $\lim\limits_{b\to\infty}\|w\|_{L^b(B_o(3R/4))}=\|w\|_{L^\infty(B_o(3R/4))}$, for any $\eta>0$, there exists $\overline{b}>0$, such that for any $b\geq \overline{b}$, we have 
		$$\|w\|_{L^\infty(B_o(3R/4))}\leq \|w\|_{L^b(B_o(3R/4))}+\eta.$$
	Let $b=b_0$ and choose $b_0\geq \overline{b}$ such that \eqref{condition} holds true. Then the first conculusion follows by Lemma \ref{bd2}.
	
We now assume that $\|{\rm Ric}_-^K\|_{q,r}=0$, this means that \eqref{condition}  holds true for any $b$ large enough. Hence 
the inequality \eqref{eq12} holds true for arbitrary $b$ large enough. Thus, the last conclusion can be verified by following a standard Moser's iteration (see \cite{DD, WZ, zhaoyang}). For the completeness, we include some details here. Note that in the proof of Lemma \ref{bd2}, we have shown the inequality \eqref{eq12}. Since the second term in the left side hand of \eqref{eq12} is non-negative, we obtain
			
			\begin{equation*}
			\left( \int_{B_{0}(R)}(w^{\frac{p+b-1}{2}}\eta)^{\frac{2m}{m-2}}\right)^{\frac{m-2}{m}}\leq a_8e^{c_{2}b_{0}}V^{-\frac{2}{m}}\int_{B_{0}(R)}(bR^{2}|\nabla \eta|^{2}+b_{0}^{2}b^{2}\eta^{2}) w^{p+b-1}.
			\end{equation*}
			To use the Moser's iteration, we put
			\begin{equation*}
			b_{\ell+1}=b_{\ell}\dfrac{m}{m-2}, \quad b_{1}=(b_0+p-1)\frac{m}{m-2},\quad \Omega_{\ell}=B_{0}(\dfrac{R}{2}+\dfrac{R}{4^\ell}), \quad \ell=1,2...
			\end{equation*}
			and choose $ \eta_{\ell}\in C_{0}^{\infty}(R) $ such that
			\begin{equation*}
			\eta_{\ell} \equiv 1 \text{ on } \Omega_{\ell+1}, \quad \eta_{l}\equiv \text{ on } B_0(R) \setminus \Omega_{\ell}, \quad |\nabla \eta_{\ell}|\leq\dfrac{C4^{\ell}}{R},\quad 0\leq\eta_{\ell}\leq 1.
			\end{equation*}
			With the above choosing and note that $b=b_0$, we have 
			\begin{equation*}
			\left( \int_{\Omega_{\ell+1}}w^{b_{\ell+1}}\right)^{\frac{1}{b_{\ell+1}}}\leq \left( a_8e^{c_{2}b_{0}}V^{-\frac{2}{m}}\right) ^{\frac{1}{b_{\ell}}}\left( \int_{\Omega_{\ell}} (b_{0}^{2}b^{2}+bR^{2}|\nabla \eta|^{2}) w^{b_{\ell}}\right) ^{\frac{1}{b_{\ell}}}.
			\end{equation*}
			A standard argument implies
			\begin{equation*}
			\parallel w \parallel_{L^{\infty}(B_{0}(\frac{R}{2}))} \leq \left( a_8e^{c_{2}b_{0}}V^{-\frac{2}{m}}\right) ^{\frac{m}{2b_{1}}} 17^{\frac{m^{2}}{4b_{1}}}(b_{0}b)^{\frac{m}{b_{1}}} \parallel w \parallel_{L^{b_{1}}(B_{0}(\frac{3R}{4}))}.
			\end{equation*}
This togather with Lemma \ref{bd2} infers 
			\begin{equation*}
			\parallel w \parallel_{L^{\infty}(B_{0}(\frac{R}{2}))}
			\leq a_9\left( \dfrac{b_{0}}{R}\right)^{2}.
			\end{equation*}
			Since $ b_{0}=c_{1}(1+\sqrt{K}R) $, we have
			\begin{equation*}
			\parallel w \parallel_{L^{\infty}(B_{0}(\frac{R}{2}))}
			\leq a_{10}\left( \dfrac{1+\sqrt{K}R}{R}\right)^{2}.
			\end{equation*}
			Since $ w=\left( \dfrac{|\nabla u|}{u}(p-1)\right) ^{2} $, we are done
		\end{proof}
		\begin{remark}If $k(q,1)\not=0$ then the condition \eqref{condition} can not satisfy for $b$ large enough. Hence, the Moser iteration can not be applied in this case. This explains why we need to add the constant $\eta>0$ in the right hand side of \eqref{kq1}.
		\end{remark} 
\section{Liouville theorems and local gradient estimates}
	In this	 section, we will point out applications of Theorem \ref{dl1} to derive some Liouville results and local gradient estimates on Riemannian manifold. Recall that $h(v)=(p-1)^{p-1}e^{-v}F(e^{\frac{v}{p-1}})$. Hence
	$$h'(v)=(p-1)^{p-1}e^{-v}\left[\frac{F'(e^{v/(p-1)})e^{v/(p-1)}}{p-1}-F(e^{v/(p-1)})\right]$$
	 First, we give a proof of Corollary \ref{hq1}.

\noindent
\begin{proof}[Proof of Corollary \ref{hq1}]
For $ F(u)=u(1-u^{2}) $ then $ F'(u)=1-3u^{2} $. It is easy to see that for $0< u\leq 1, p\geq2$ then $v=\log u\leq0$, consequently $0<e^{v/(p-1)}\leq 1$. Moreover, if $0<u\leq1$ then
		$$\begin{aligned}
		\dfrac{F'(u)u}{p-1}-F(u)&=\dfrac{(1-3u^{2})u}{p-1}- u(1-u^{2})\\
		&=\frac{u}{p-1}((p-4)u^2-(p-2))\leq0.
		\end{aligned}$$	
		Hence, $h'(v)\leq 0$ assumpsion of Theorem \ref{dl1} holds. So we have  \eqref{main}. When $ K=0 $, this implies
		\begin{equation*}
		\dfrac{|\nabla u|}{u}\leq \dfrac{C_{p,m}}{R}.
		\end{equation*}
		Let $ R \rightarrow +\infty $, since $ u>0 $ then we have $ \nabla u=0 $, therefore  $ u $ is constant on $ M $. This leads to $ \Delta_{p,f}u=0 $, as a consequence, we have $ u(1-u^{2})=0 $. Using condition $ 0<u\leq 1 $, we conclude $ u=1 $ on $ M $. The proof is complete.
\end{proof}				
\begin{proof}[Proof of Corollary \ref{hq2}] By assumption we have $F(u)=cu(1-u)=cu-cu^2$. Therefore, for $0< u\leq 1, p\geq2$ then
		$$\begin{aligned}
		\dfrac{F'(u)u}{p-1}-F(u)&=\dfrac{(c-2cu)u}{p-1}- cu-cu^{2}\\
		&=\frac{cu}{p-1}((p-3)u-(p-2))\leq0.
		\end{aligned}$$	
The proof follows directly from Therem \ref{dl1}.
\end{proof}
\begin{proof}[Proof of Corollary \ref{hq31}]We have $F(u)=au\log u$. Hence for $p=2, v=\log u\geq0$, we have $h(v)=av\geq0$ and $h'(v)=a\geq0$. The proof follows directly from Therem \ref{dl1}.
\end{proof}
		\begin{remark}
Using the same argument as in proof of Corollaries \ref{hq1}-\ref{hq31}, we can obtain gradient estimates for the Lichnerowicz type equation $
		\Delta_{p,f}u+u^{a}-u^{b}=0,
		$
on a smooth metric space with $ Ric_{f}^{m}\geq-(m-1)K $, $ K $ is a non-negative constant. We leave the details of computations for the reader.
\end{remark}
Finally, we introduce a local gradient estimate for a nonlinear equation under integral Ricci curvature condition.
\begin{corollary} \label{hq311}
		Let $ (M,g) $ be complete Riemannian manifold. Suppose that $ u\geq1 $ is a  positve solution of equation
		\begin{equation}\label{soliton}
		\Delta_{f}u+au\log u=0, \quad a\geq0,
		\end{equation}
		on the geodesic ball $ B_{0}(R)\subset M $. For $q>n/2$ and $R\leq1$, then for any $\eta>0$ there exists $b$ large enough such that if $k(q,1)\leq \frac{1}{b}$ and $\|{\rm Ric}_{-}^K\|_{q,r}\leq\frac{1}{bR^2}$ then 
		\begin{equation*} 
		\dfrac{|\nabla u|}{u}\leq C_{p,m, V}\dfrac{1+\sqrt{K}R}{R} + \eta
		\end{equation*}
		on the geodesic ball  $ B_{0}(\frac{R}{2}) $, with $ C_{p,m, V} $ only depends on $ p, m $ and $ V=V(B_o(R)) $.
	\end{corollary} 
	When $K=0$, we have $k(q,1)=\|{\rm Ric}_{-}^K\|_{q,r}$, the above corollary can be stated as follows.
	\begin{corollary}
		Let $ (M,g) $ be complete Riemannian manifold. Suppose that $ u\geq1 $ is a  positve solution of equation
		\begin{equation}\label{soliton}
		\Delta_{f}u+au\log u=0, \quad a\geq0,
		\end{equation}
		on the geodesic ball $ B_{0}(R)\subset M $. For $q>n/2$ and $R\leq1$, then for any $\eta>0$ there exists $b$ large enough such that if $k(q,1)\leq \frac{1}{b}$ then 
		\begin{equation*} 
		\dfrac{|\nabla u|}{u}\leq C_{p,m, V}\dfrac{1+\sqrt{K}R}{R} + \eta
		\end{equation*}
		on the geodesic ball  $ B_{0}(\frac{R}{2}) $, with $ C_{p,m, V} $ only depends on $ p, m $ and $ V=V(B_o(R)) $.
	\end{corollary} 
To prove Corollary \ref{hq311}, we need to use the following local Sobolev inequality (see Corollary 4.6 in \cite{Daietal}).
\begin{lemma}[\cite{Daietal}]\label{sobolevintegral}
For any $q > n/2$, there exists $\varepsilon = \varepsilon(p, n) >0 $ such that if $M^n$ has $k(p, 1) \leq  \varepsilon$, then for any $o \in M, r \leq 1$, we have
$$\left( \int_{B_{0}(R)}|\phi|^{\frac{2m}{m-2}}\right) ^{\frac{m-2}{m}}\leq C(n)V^{-2/n}\int_{B_{0}(R)}(R^{2}|\nabla \phi|^{2}+\phi^{2}),$$
where $V=V(B_0(R))$.
\end{lemma}
\begin{proof}[Proof of Corollary \ref{hq311}]Since $\|{\rm Ric}_{-}^K\|_{p,r}\leq\frac{1}{bR^2}$, the condition \eqref{condition} holds true for $b$ large enough. We can assume that such $b$ to be satisfied $\frac{1}{b}\leq \varepsilon$. Combining the assumption $k(p,1)\leq\frac{1}{b}$ and Lemma \ref{sobolevintegral}, we conclude that $M$ has a Sobolev inequality. Therefore, the proof follows directly from Theorem \ref{dl1}.
\end{proof}

	\end{document}